\documentclass[12pt,leqno,twoside,sort]{amsart}

\usepackage{amssymb,amsmath,amsthm,soul,color}

\usepackage[normalem]{ulem}

\usepackage{amssymb,amsmath,amsthm}

\usepackage{mathrsfs}

\usepackage[utf8]{inputenc}

\usepackage{todonotes}

\usepackage[left=2cm, right=2cm, top=2cm, bottom=2cm]{geometry}

\usepackage{url}

\usepackage[numbers,sort&compress]{natbib}

\usepackage{fixmath}

\usepackage[T1]{fontenc}
\usepackage{lmodern}
\usepackage{microtype}
\usepackage[normalem]{ulem}

\usepackage[colorlinks=true,urlcolor=blue,
citecolor=red,linkcolor=blue,linktocpage,pdfpagelabels,
bookmarksnumbered,bookmarksopen]{hyperref}

\usepackage{xcolor,cancel}

\newtheorem{Th}{Theorem}[section]
\newtheorem{Prop}{Proposition}[section]
\newtheorem{Lem}{Lemma}[section]

\theoremstyle{definition}

\newtheorem{Rem}{Remark}[section]

\newcommand{\C}{\mathbb{C}}
\newcommand{\R}{\mathbb{R}}

\newcommand{\cA}{{\mathcal A}}

\newcommand{\cC}{{\mathcal C}}

\renewcommand{\div}{\mathrm{div}\,}

\newcommand{\tu}{\widetilde{u}}
\newcommand{\tv}{\widetilde{v}}

\numberwithin{equation}{section}

\begin{document}

\title{Magnetostatic levitation and two related linear PDEs in unbounded domains}

\author[B. Bieganowski]{Bartosz Bieganowski}

\author[T. Cieślak]{Tomasz Cieślak}

\author[J. Siemianowski]{Jakub Siemianowski}

\address[B. Bieganowski and J. Siemianowski]{\newline\indent  	Institute of Mathematics,		\newline\indent 	Polish Academy of Sciences, \newline\indent ul. \'Sniadeckich 8, 00--656 Warszawa, Poland
\newline \indent and \newline
\indent Faculty of Mathematics and Computer Science,
	\newline\indent 
	Nicolaus Copernicus University in Toru\'{n} 
	\newline\indent 
	ul. Chopina 12/18, 87-100 Toruń, Poland}
\email{\href{mailto:bbieganowski@impan.pl}{bbieganowski@impan.pl}, \href{mailto:bartoszb@mat.umk.pl}{bartoszb@mat.umk.pl}}
\email{\href{mailto:jsiemianowski@impan.pl}{jsiemianowski@impan.pl}, \href{mailto:jsiem@mat.umk.pl}{jsiem@mat.umk.pl}}	

\address[T. Cieślak]{\newline\indent  	Institute of Mathematics,		\newline\indent 	Polish Academy of Sciences, \newline\indent ul. \'Sniadeckich 8, 00--656 Warszawa, Poland}
\email{\href{mailto:cieslak@impan.pl}{cieslak@impan.pl}}

\begin{abstract} 
We consider a problem occurring in a magnetostatic levitation. The problem leads to a linear PDE in a strip. In engineering literature a particular solution is obtained.
Such a solution enables one to compute lift and drag forces of the levitating object. It is in agreement with the experiment.
We show that such a solution is unique in a class of bounded regular functions.
Moreover, as a byproduct, we obtain nonstandard uniqueness results in two linear PDEs in unbounded domains.
One of them is an Eigenvalue problem for the Laplacian in the strip in the nonstandard class of functions.

\medskip

   
\end{abstract}

\maketitle

\section{Introduction}

Magnetic levitation is currently a hot topic. All around the world there are attempts to use it in an ecological transport. Maglev (from magnetic levitation) and similar technologies are believed to offer new ways of transportation. Those include, for instance trains in Japan (MLX01), commercial one in Shanghai (from the city centre to the airport) and airport lines in Incheon, South Korea. 

In the present article we are concerned with a simple model of magnetostatic levitation. 
We follow an exposition in \cite{RJH}. The author describes there a straightforward analytical approach 
to maglev. The analytical model is then compared with the results of experiments. 

Let us begin with the presentation of the model  studied in \cite{RJH} with some justification.
We consider the standard Cartesian $xyz$ coordinates and a conducting plate moving parallel to a sinusoidally-distributed current sheet $\mathbf{J}=(0,0,J_0 \cos(x/\lambda))$ with $\lambda$ being the current sheet wavelength.
The conducting plate has a uniform thickness and conductivity $\sigma$.
We assume that the plate bottom surface is located in the $\{y=0\}$ plane and has a constant velocity in the positive $x$ direction $\mathbf{v}=(v_x,0,0)$.
The system is described (see \cite{CB, RJH, KOF}) by the magnetostatic case of Maxwell's equations
\[
\nabla\cdot  \mathbf{B}=0, \;\; \nabla\times \mathbf{B}=\mu_0 \mathbf{J},
\]
where $\mathbf{B}=(B_x, B_y, B_z)$ stands for the magnetic field and  $\mu_0$ is the vacuum permeability.
The movement of the plate through the magnetic field induces the motional electromotive force and the electric field $\mathbf{E}$ which is given by
\[
\mathbf{E} = \mathbf{v}\times \mathbf{B},
\]
see \cite[Chapter 4]{KOF}. We incorporate Ohm's law
\[
\mathbf{J} = \sigma \mathbf{E}.
\]
If we combine the above equations, we see that
\[
\nabla \cdot \mathbf{B} = 0\quad \text{and} \quad \nabla \times \mathbf{B} = \mu_0 \sigma (\mathbf{v}\times \mathbf{B}).
\]
We study the system in the two dimensions only by assuming that the magnetic field $\mathbf{B}$ does not depend on the $z$ coordinate and $B_z = 0$, and all derivatives with respect to $z$ are zero.
This simplification implies that 
\begin{equation}\label{pomocnicze}
\frac{\partial B_x}{\partial x}=-\frac{\partial B_y}{\partial y}
\end{equation}
and also that only the last coordinate in $\nabla \times \mathbf{B} = \mu_0 \sigma (\mathbf{v}\times \mathbf{B})$ is nonzero
\[
\frac{\partial B_y}{\partial x} - \frac{\partial B_x}{\partial y} = \mu_0 \sigma v_x B_y.
\]
Differentiating both sides with respect to $x$ and using \eqref{pomocnicze} give rise to the following elliptic equation
\begin{equation}\label{rownanie}
\Delta B_y=\mu_0\sigma v_x \frac{\partial B_y}{\partial x},
\end{equation}
which is a starting problem for the forthcoming mathematical analysis.

An author of \cite{RJH} finds a solution of (\ref{rownanie})
of the form 
\begin{equation}\label{rozw}
B_y= B_0 \exp(-\alpha y)\sin\left[(x/\lambda)+by\right].
\end{equation}
Notice that such a solution is not of separated variables form, it seems simply guessed.
By (\ref{pomocnicze}), one computes $B_x$ and, in the end, both drag and lift forces
are calculated using solution $(B_x, B_y)$. They are compared with the experimentally obtained ones.
If the plate is thin enough that the edge effects as well as eddy currents can be neglected, 
analytical solutions are in agreement with experimentally obtained data.

There are two questions. How does one arrive at a solution of this particular form? And, more importantly, is it a unique solution to (\ref{rownanie})? 
The second one can be treated in the following way.
 Consider (\ref{rownanie}) in a domain being an infinite 
two-dimensional strip (this is our model of a plate) under the Dirichlet boundary conditions on the top and a bottom of a strip.
Is the solution unique? The equation is linear, however the domain, an infinite strip, suggests that the problem is not trivial.
Our paper is devoted to answering the above question. To this end, we use some modification of a particular maximum principle, going back to
E. Hopf. Section \ref{sect:2} is devoted to this problem. Last, but not least, we observe that our problem is related to other two interesting 
linear PDEs. On the one hand, a particular degenerate problem in a half-space, not satisfying the $\mathcal{A}$-harmonicity (cf. \cite{He}) is considered
in Section \ref{sect:3}. We prove the correspondence between the problem in a half-space and our original problem in a strip. This allows us to state the uniqueness results concerning the degenerate problem in a half-space in a space of regular functions. The short Section \ref{sect:4} is devoted to the Eigenproblem of a Laplacian in a strip, in a nonstandard class of functions. Finally, last section presents some duality results 
in the case of a degenerated operator introduced in Section \ref{sect:3}. 

In what follows, $\lesssim$ denotes the inequality up to a multiplicative constant.

\section{The maximum principle and uniqueness of solutions}\label{sect:2}

In what follows we consider the equation \eqref{rownanie} with $k := \mu_0 \sigma v_x$ in a strip. For convenience we consider a strip of width $\pi$. Since the solution \eqref{rozw} is only bounded (does not vanish as $x \to \pm \infty$ and is periodic in $x$), we are interested in a uniqueness result in the class of bounded functions.

Let $\Omega := \mathbb{R} \times (0, \pi)$ and $k$, $\lambda$ be reals with $\lambda \geq 0$.
The following theorem comes in fact from \cite{G} and \cite{H}.
There are no explicit results for the differential operator $L := \Delta - k\partial_x -\lambda$ and the strip domain.
Nevertheless, there is a comment on the top of the page 420 of \cite{H} that the Phragm\'{e}n-Lindel\"of principle is valid in such a case.
We present the proof with suitable adjustments for the sake of the clarity.

\begin{Th}\label{MaxPrinciple}
Let $\Omega = \mathbb{R} \times (0, \pi)$, $k \in \R$ and $\lambda \geq 0$. Suppose that $u \in \cC^2 (\Omega) $ is a solution to
\[
\Delta u - k u_x - \lambda u = 0 \quad \mbox{in} \  \Omega
\]
and $\liminf_{(x^\prime,y^\prime)\to (x,y)}u(x^\prime,y^\prime) \geq 0$, for every $(x,y) \in \partial \Omega$.
Let $\mu_R := \displaystyle \min_{x^2+y^2=R^2,\ (x,y) \in \Omega} u(x,y)$.
If
$$
\lim_{n \to +\infty} \frac{\mu_{R_n}}{R_n} = 0, \ \mbox{for some } R_n\to +\infty,
$$
then $u \geq 0$ in $\Omega$.
\end{Th}

\begin{proof}
Let $B_R := \{ (x,y) \in \R^2 \mid x^2 + y^2 \leq R^2 \}$. We shall construct a family of functions $v_R \in \cC^2 ( \Omega \cap B_R )$, $R>\pi$,  which is continuous on $\overline{\Omega \cap B_R} \setminus \{ (-R,0), (R,0) \}$ and satisfies:
\begin{itemize}
\item[(i)] $v_R = 0$ for $y = 0$, $v_R = 1$ for $x^2+y^2=R^2$ and $y > 0$, and $0 < v_R \leq 1$ for $y=\pi$,
\item[(ii)] $0 < v_R < 1$ in $\Omega \cap B_R$ and $L[v_R] \leq 0$ in $\Omega \cap B_R$,
\item[(iii)] for every $(x,y) \in \Omega \cap B_R$, $R v_R (x,y)$ is bounded as $R \to +\infty$.
\end{itemize}

\begin{figure}
  \includegraphics[width=\linewidth]{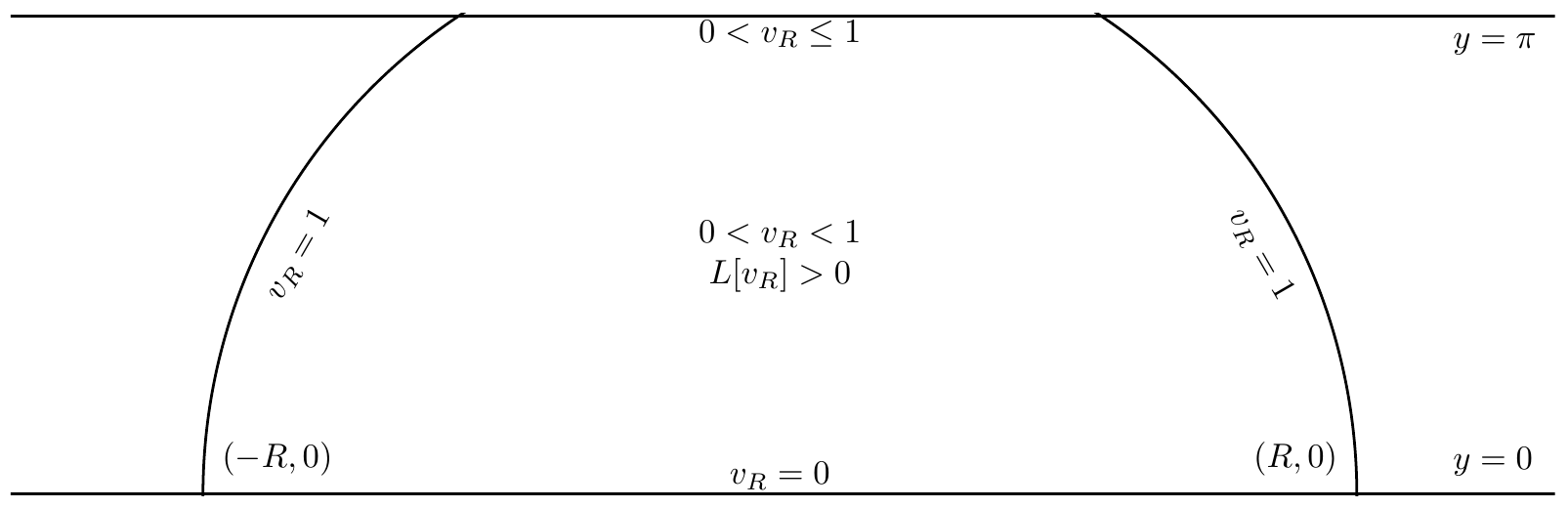}
  \caption{Properties of $v_R$.}
  \label{fig:1}
\end{figure}

For now, suppose that a function $v_R$ satisfying (i)--(iii) exists. Let $\sigma_R := \min \{0, \mu_R \} \leq 0$ and consider the function
$$
u_R (x,y ) := u(x,y) - \sigma_R v_R (x,y),\quad (x,y)\in \Omega\cap B_R.
$$
Then $u_R \in \cC^2 (\Omega \cap B_R)$ and the hypothesis on $u$ and the condition (i) guarantee that $u_R\geq 0$ on $\partial (\Omega\cap B_R)$ in the sense of the lower limit.
By (ii), we have
$$
L[u_R] = L[u] - \sigma_R L[v_R] = -\sigma_R L[v_R] \leq  0, \quad\text{in} \ \Omega \cap B_R.
$$
The maximum principle on $\Omega \cap B_R$ yields  $u_R \geq 0$ in $\Omega \cap B_R$ or, equivalently,
$$
u \geq \frac{\sigma_R}{R} R v_R\quad \text{in} \ \Omega \cap B_R.
$$
Let us fix $(x,y) \in \Omega$. There is $n_0$ such that $(x,y) \in \Omega \cap B_{R_n}$ for every $n \geq n_0$, and so
$$
u(x,y) \geq \frac{\sigma_{R_n}}{R_n} R_n v_{R_n} (x,y) = \min \left\{0, \frac{\mu_{R_n}}{R_n} \right\} R_n v_{R_n} (x,y)
$$
for $n \geq n_0$. Passing to the limit as $n \to + \infty$ and using (iii), we deduce that $u \geq 0$ in $\Omega$.

We have shown it is sufficient to construct $v_R$ satisfying (i)--(iii). Consider the function $h : \Omega \cap B_1 \rightarrow \R$ given by
$$
h(x,y) := \frac{2}{\pi} \left( \arctan \frac{y}{x+1} - \arctan \frac{y}{x-1} \right),\quad (x,y)\in \Omega\cap B_1.
$$
Geometrically, $h$ is represented by the following formula
$$
h(x,y) = \frac{2}{\pi} (\pi - \theta(x,y)) ,
$$
where the angle $\theta(x,y)$ is shown on Fig. \ref{fig:2}.

\begin{figure}
  \includegraphics[width=\linewidth]{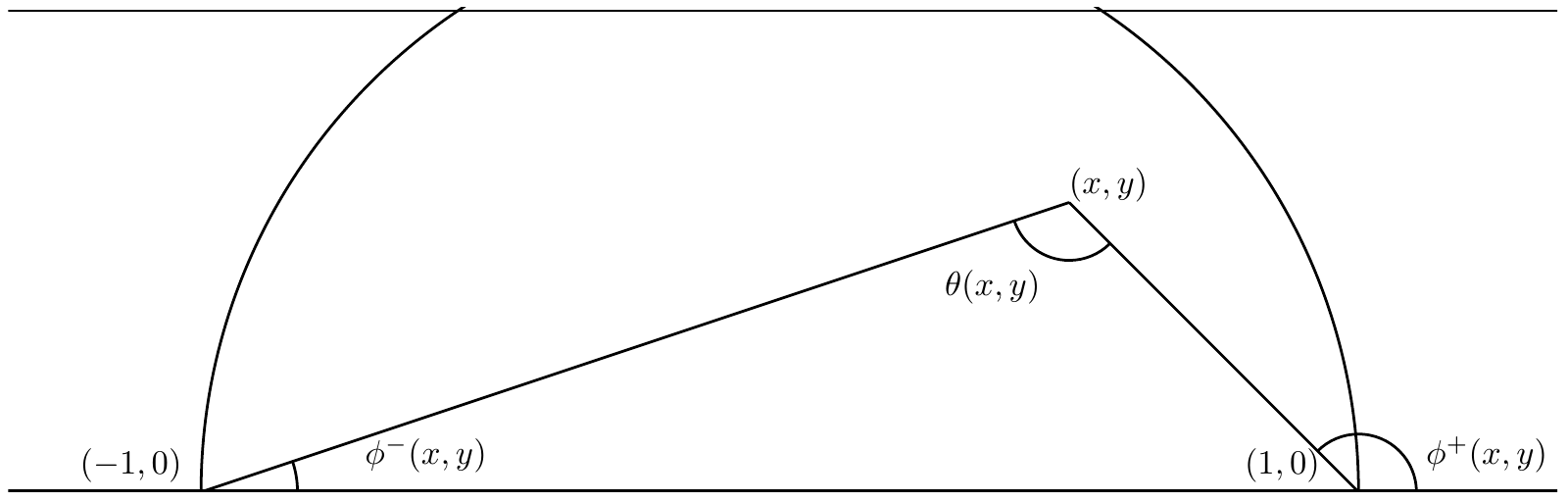}
  \caption{Definition of $\theta$.}
  \label{fig:2}
\end{figure}

We note the following straightforward properties of $h$:
\begin{itemize}
\item $h$ is harmonic in $\Omega\cap B_1$,
\item $h(x,0)=0$,
\item $h(x,y) = 1$ for $x^2+y^2=1$, $y>0$,
\item $0 < h <  1$ in $\Omega \cap B_1$,
\item $h$ is smooth on $\Omega \cap B_1$ and continuous on $\overline{\Omega \cap B_1} \setminus \{ (-1,0), (1,0) \}$.
\end{itemize} 
We define $v_R : \Omega \cap B_R \rightarrow \R$ by
$$
v_R(x,y) := f \left( h \left( \frac{x}{R}, \frac{y}{R} \right) \right),
$$
where $f : \R \rightarrow \R$ is given by
$$
f(t) := \frac{\int_0^t e^{-\frac{\pi^2 |k| \zeta}{2}} \, d\zeta}{\int_0^1 e^{-\frac{\pi^2 |k| \zeta}{2}} \, d\zeta}.
$$
Since $f$ is smooth, $v_R \in \cC^2 (\Omega \cap B_R)$ and $v_R$ is continuous on $\overline{\Omega \cap B_R} \setminus \{ (-R,0), (R,0) \}$.
Obviously $f(0) = 0$, $f(1)=1$ and $f(t) \in (0,1)$, for $t \in (0,1)$, and thus $v_R$ satisfies (i) and the first part of (ii). 

For $(x,y) \in \Omega \cap B_R$ we have
$$
R v_R(x,y) = \underbrace{\frac{f \left( h \left( \frac{x}{R}, \frac{y}{R} \right) \right) - f(0)}{h \left( \frac{x}{R}, \frac{y}{R} \right) - h(0,0)}}_{\to f^\prime(0)} \cdot \underbrace{\frac{h \left( \frac{x}{R}, \frac{y}{R} \right)}{\frac{1}{R}}}_{\to \frac{4y}{\pi}}\to \frac{4y}{\pi}f^\prime(0),\quad \text{as}\ R\to+\infty,
$$
where we used L'H\^{o}pital's rule to calculate the second term.
Hence $R v_R$ stays bounded as $R \to +\infty$ and (iii) holds. 
The only thing left is to check that 
$$
L[v_R] \leq 0 \mbox{ on } \Omega \cap B_R.
$$
We exploit the fact that $h$ is harmonic to get
\begin{align*}
(\Delta - k \partial _x)[v_R(x,y)] &= \frac{1}{R^2}f'' \left( h \left( \frac{x}{R}, \frac{y}{R} \right) \right) \left| \nabla h \left( \frac{x}{R}, \frac{y}{R} \right) \right|^2 - \frac{k}{R} f' \left( h \left( \frac{x}{R}, \frac{y}{R} \right) \right) h_x  \left( \frac{x}{R}, \frac{y}{R} \right).
\end{align*}
Since $f'  > 0$ and $\left| \nabla h \left( \frac{x}{R}, \frac{y}{R} \right) \right|^2 > 0$, for $(x,y) \in \Omega \cap B_R$, we divide
\begin{equation}\label{eq:1}
\frac{(\Delta - k \partial _x)[v_R(x,y)]}{\frac{1}{R^2} f' \left( h \left( \frac{x}{R}, \frac{y}{R} \right) \right) \left| \nabla h \left( \frac{x}{R}, \frac{y}{R} \right) \right|^2} = \frac{f'' \left( h \left( \frac{x}{R}, \frac{y}{R} \right) \right)}{f' \left( h \left( \frac{x}{R}, \frac{y}{R} \right) \right)} - \frac{k R h_x  \left( \frac{x}{R}, \frac{y}{R} \right)}{\left| \nabla h \left( \frac{x}{R}, \frac{y}{R} \right) \right|^2}.
\end{equation}
Denoting $A := (R+x)^2 + y^2$ and $B := (R-x)^2 + y^2$, we obtain
\begin{align*}
h_x \left( \frac{x}{R}, \frac{y}{R} \right) = \frac{2}{\pi} R \left( \frac{-y}{A} + \frac{y}{B} \right),\qquad
h_y \left( \frac{x}{R}, \frac{y}{R} \right) = \frac{2}{\pi} R \left( \frac{x+R}{A} + \frac{-x+R}{B} \right)
\end{align*}
and therefore
\begin{align*}
\left| \nabla h \left( \frac{x}{R}, \frac{y}{R} \right) \right|^2 &= \frac{4}{\pi^2} R^2 \left(\frac{1}{A} + \frac{1}{B} + 2\frac{R^2-x^2-y^2}{AB} \right) > \frac{4}{\pi^2} R^2 \left(\frac{1}{A} + \frac{1}{B} \right), \quad \mbox{for } (x,y) \in \Omega \cap B_R.
\end{align*}
It follows that
\begin{equation}\label{y-bounded}
\frac{\left| k R h_x  \left( \frac{x}{R}, \frac{y}{R} \right) \right|}{\left| \nabla h \left( \frac{x}{R}, \frac{y}{R} \right) \right|^2} \leq \frac{|k| R^2 \frac{2}{\pi} \left( \frac{|y|}{A}+\frac{|y|}{B} \right)}{\frac{4}{\pi^2} R^2 \left( \frac{1}{A}+\frac{1}{B} \right)} = \frac{|k| \pi |y|}{2} \leq  \frac{|k| \pi^2}{2}.
\end{equation}
Observe that
$$
\frac{f''(t)}{f'(t)} = - \frac{\pi^2 |k|}{2},
$$
so \eqref{eq:1} is nonpositive and, consequently, $(\Delta - k \partial _x)[v_R]\leq 0$ in $\Omega \cap B_R$.
Finally, we use the fact that $v_R > 0$ in $\Omega$ to obtain
\[
L[v_R]  = (\Delta - k \partial _x)[v_R] - \lambda v_R \leq 0
\]
what completes the proof.
\end{proof}

\begin{Rem}
If the differential operator in the equation in the statement of Theorem \ref{MaxPrinciple} would also include the first derivative with respect to $y$, we would also have terms like $|x|$ in \eqref{y-bounded} which couldn't be bounded in such an easy manner.
\end{Rem}

As an immediate consequence of Theorem \ref{MaxPrinciple}, we obtain the following uniqueness result in the class of functions with sublinear growth.

\begin{Th}\label{Thm:uniqueness}
Let $\Omega = \R \times (0,\infty)$, $f:\Omega\to\R$ and $g$, $h:\R\to\R$ be continuous, $k$ be an arbitrary real constant. If $\lambda \geq 0$, then the problem 
\begin{equation}\label{problem}
\left\{ \begin{array}{ll}
\Delta u - k u_x -\lambda u= f(x,y), & \quad \mbox{in} \ \Omega, \\
u(x, 0) = g(x), \ u(x,\pi)=h(x), & \quad x \in \R
\end{array} \right.
\end{equation}
has at most one solution in the class $\left\{u\in \cC^2(\Omega)\cap\cC(\overline{\Omega})\mid \lim_{\sqrt{x^2+y^2}\to+\infty} \frac{u(x,y)}{\sqrt{x^2+y^2}} = 0\right\}$.
In~particular, there is at most one solution in the class $\cC^2(\Omega) \cap \cC_b(\overline{\Omega})$.
\end{Th}

\begin{proof}
If $u_1, u_2$ are solutions to \eqref{problem} in the class $\left\{u\in \cC^2(\Omega)\cap\cC(\overline{\Omega})\mid \lim_{\sqrt{x^2+y^2}\to+\infty} \frac{u(x,y)}{\sqrt{x^2+y^2}} = 0\right\}$, we apply Theorem \ref{MaxPrinciple} to $u_1-u_2$ and to $u_2-u_1$ to get that $u_1 - u_2 \equiv 0$.
\end{proof}

\begin{Rem}
Let us emphasize that the growth condition (or boundedness) on the solution is essential.
Consider the homogeneous equation with null boundary data
\[
\begin{cases}
\Delta u -ku_x =0,&\text{in }\Omega\\
u(x,0) = u(x,\pi) = 0 &x\in \R.
\end{cases}
\]
One solution is obviously $u \equiv 0$ but the separation of variables method yields other solutions
\[
u_l(x,y) = \left ( A e^{\frac{k+\sqrt{k^2+4l^2}}{2}x} + B e^{\frac{k-\sqrt{k^2+4l^2}}{2}x} \right ) \sin l y,\quad l \in \mathbb{N}, \,A,\, B \in \mathbb{R}.
\]
However those solutions grow exponentially as $x\to \pm \infty$.
\end{Rem}

\section{The problem on the punctured half-space}\label{sect:3}

As we will see in a moment, the problem \eqref{problem} is connected to the following equation
$$
\div ( |\mathbf{x}|^{-k} \nabla_\mathbf{x} \tu) = 0, \quad \mathbf{x} \in \R_+^2 := \R \times (0,+\infty).
$$
Let us recall that if $\cA : U \times \R^N \rightarrow \R^N$, where $U$ is an open subset of $\R^N$, is a Carath\'eodory function\footnote{Here it means that $\cA(\cdot, \xi)$ is measurable for every $\xi \in \R^N$ and $\cA(\mathbf{x}, \cdot)$ is continuous for a.e. $\mathbf{x} \in U$.}, then any solution to
$$
\div \cA(\mathbf{x}, \nabla \tu) = 0, \quad \mathbf{x} \in U
$$
is called a $\cA$-harmonic map in $U$. Usually, one needs to assume additionally the following growth conditions
\begin{equation}\label{A-growth}
\cA(\mathbf{x},\xi) \cdot \xi \gtrsim w(\mathbf{x}) |\xi|^p, \quad |\cA(\mathbf{x},\xi)| \lesssim w(\mathbf{x}) |\xi|^{p-1}, \quad \mathbf{x} \in U, \ \xi \in \R^N,
\end{equation}
where $p > 1$ and $w$ is a $p$-admissible weight (see \cite[Section 3, Section 6]{He}). In our case $\cA(\mathbf{x}, \xi) = |\mathbf{x}|^{-k} \xi$ does satisfy \eqref{A-growth} with $p=2$, $w(\mathbf{x}) = |\mathbf{x}|^{-k}$ and $k < 2$ (see \cite[Corollary 15.35]{He}). Therefore, Theorem \ref{thm:1} may be seen as a generalization of uniqueness results of \cite[Section 3]{He} to the case of the unbounded domain $\R_+^2$.

From Theorem \ref{Thm:uniqueness}, we know that $u \equiv 0$ is the unique solution to
\begin{align*}
\left\{ \begin{array}{ll}
\Delta u - k u_x=0, & \quad \mbox{in } \Omega, \\
u(x,0) = u(x,\pi) = 0, & \quad x \in \R
\end{array} \right.
\end{align*}
in $\cC^2(\Omega) \cap \cC_b(\overline{\Omega})$.

Let us set $z := x + iy$ and treat $\Omega \subset \C$ as a subset of the complex plane $\C$. We introduce the conformal mapping $z \mapsto w = e^z$ (see Figure \ref{fig:3}) which transforms a strip onto a half-space
$$
\Omega \ni z \mapsto w \in \R_+ ^2
$$
and we set $\tu(w) := u(z)$. 

\begin{figure}
  \includegraphics[width=\linewidth]{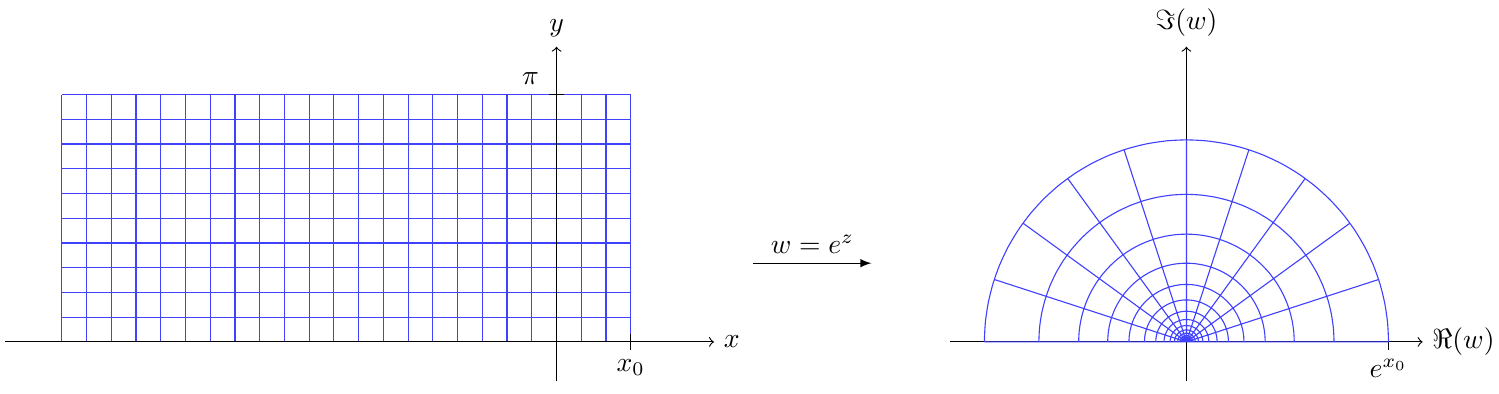}
  \caption{The map $w=e^z$.}
  \label{fig:3}
\end{figure}

We recall the Wirtinger derivatives $\frac{\partial}{\partial z} = \frac12 \left( \frac{\partial}{\partial x} - i \frac{\partial}{\partial y} \right)$ and $\frac{\partial}{\partial \overline{z}} = \frac12 \left( \frac{\partial}{\partial x} + i \frac{\partial}{\partial y} \right)$.
Since $w = e^z$ and $\overline{w} = e^{\overline{z}}$ we obtain
\begin{align*}
\frac{\partial u}{\partial z} = \frac{\partial \tu}{\partial w} \frac{\partial w}{\partial z} + \frac{\partial \tu}{\partial \overline{w}} \frac{\partial \overline{w}}{\partial z} = \frac{\partial  \tu}{\partial w} w, \\
\frac{\partial u}{\partial \overline{z}} = \frac{\partial \tu}{\partial w} \frac{\partial w}{\partial \overline{z}} + \frac{\partial \tu}{\partial \overline{w}} \frac{\partial \overline{w}}{\partial \overline{z}} =  \frac{\partial \tu}{\partial \overline{w}} \overline{w}
\end{align*}
Hence
$$
\frac{\partial^2 u}{\partial \overline{z} \partial z} = \frac{\partial}{\partial \overline{z}} \left( \frac{\partial \tu}{\partial w} w \right) = \frac{\partial^2 \tu}{\partial \overline{w} \partial w} |w|^2.
$$
Thus we can rewrite $\Delta u - k u_x = 0$ as
$$
4 \frac{\partial^2 \tu}{\partial \overline{w} \partial w} |w|^2 - k \left( \frac{\partial  \tu}{\partial w} w + \frac{\partial \tu}{\partial \overline{w}} \overline{w} \right) = 0
$$
and by using the variable $w = \tilde{x}+i\tilde{y}$ we end with
\begin{align*}
 &(\tilde{x}^2 + \tilde{y}^2) \Delta_{(\tilde{x}, \tilde{y})} \tu  - k \left( \frac12 \left( \frac{\partial \tu}{\partial \tilde{x}} - i \frac{\partial \tu}{\partial \tilde{y}} \right)(\tilde{x} + i\tilde{y}) + \frac12 \left( \frac{\partial \tu}{\partial \tilde{x}} + i \frac{\partial \tu}{\partial \tilde{y}} \right)(\tilde{x} - i\tilde{y}) \right) \\
&\qquad = (\tilde{x}^2 + \tilde{y}^2) \Delta_{(\tilde{x}, \tilde{y})} \tu - \frac{k}{2} \left( 2 \frac{\partial \tu}{\partial \tilde{x}} \tilde{x} + 2 \frac{\partial \tu}{\partial \tilde{y}} \tilde{y} \right)   =(\tilde{x}^2 + \tilde{y}^2) \Delta_{(\tilde{x}, \tilde{y})} \tu - k \left(  \frac{\partial \tu}{\partial \tilde{x}} \tilde{x} +  \frac{\partial \tu}{\partial \tilde{y}} \tilde{y} \right) = 0,
\end{align*}
Let us write $\mathbf{x} := (\tilde{x},\tilde{y})$ and transform our equation into $|\mathbf{x}|^2 \Delta_\mathbf{x} \tu - k \mathbf{x} \cdot \nabla_\mathbf{x} \tu = 0$. We can restate it as follows
$$
\div ( |\mathbf{x}|^{-k} \nabla_\mathbf{x} \tu ) = 0.
$$
Indeed, we have
$$
\div ( |\mathbf{x}|^{-k} \nabla_\mathbf{x} \tu ) = |\mathbf{x}|^{-k} \Delta_\mathbf{x} \tu - k |\mathbf{x}|^{-k-2} \mathbf{x} \cdot \nabla_\mathbf{x} \tu = |\mathbf{x}|^{-k-2} \left( |\mathbf{x}|^2 \Delta_\mathbf{x} \tu - k \mathbf{x} \cdot \nabla_\mathbf{x} \tu \right) = 0.
$$
Since $u \in \cC^2 (\Omega) \cap \cC_b(\overline{\Omega})$ if and only if $\tu \in \cC^2 (\R_+^2) \cap \cC_b(\overline{\R_+^2} \setminus \{(0,0)\})$, we obtain the uniqueness for the following problem on the punctured half-plane.

\begin{Th}\label{thm:1}
Let $k > 0$. Suppose that $\tu \in \cC^2 (\R_+^2) \cap \cC_b(\overline{\R_+^2} \setminus \{(0,0)\}) $ is a solution to
$$
\left\{ \begin{array}{ll}
\div ( |\mathbf{x}|^{-k} \nabla_\mathbf{x} \tu ) = 0, & \quad \mathbf{x} \in \R_+^2, \\
\tu(\tilde{x},0) = 0, & \quad \tilde{x} \neq 0.
\end{array} \right.
$$
Then $\tu \equiv 0$.
\end{Th}

\section{Eigenvalue problem in a strip}\label{sect:4}
Let us recall that $\Omega = \R\times (0,\pi)$.
Suppose that $v\in \cC^2 (\Omega) \cap \cC(\overline{\Omega})$ is a solution of the Eigenvalue problem
\[
\Delta v = \lambda v \quad \mbox{in } \Omega, 
\]
for some fixed $\lambda\geq 0$. Then the function $u_a(x,y) = v(x,y)e^{ax}\in \cC^2 (\Omega) \cap \cC(\overline{\Omega})$, for $a\in \R$, is a solution of
\[
\Delta u_a - 2a\partial_x u_a - (\lambda - a^2)u = 0.
\]
Notice that for $a \in \left [ -\sqrt{\lambda} , \sqrt{\lambda} \right ]$ the term $\lambda - a^2 $ is nonnegative.
If we assume that $|v|\lesssim e^{-a x}$, then $u_a$ is bounded for $a \in \left [ -\sqrt{\lambda} , \sqrt{\lambda} \right ]$.

As a consequence of Theorem \ref{Thm:uniqueness}, we obtain the following.

\begin{Prop}
Let $\lambda \geq 0$.
There is no Eigenfunction to the problem
\begin{align*}
\left\{ \begin{array}{ll}
\Delta v = \lambda v, & \quad \mbox{in } \Omega, \\
v(x,0)=v(x,\pi)=0, & \quad x \in \R,
\end{array} \right.
\end{align*}
in the class $\left\{ v \in \cC^2 (\Omega) \cap \cC (\overline{\Omega}) \ : \ \mbox{there exists } a \in \left [ -\sqrt{\lambda} , \sqrt{\lambda} \right ] \mbox{ such that } |v(x,y)| \lesssim e^{a x} \right\}$.
\end{Prop}

\section{Appendix: Duality}

Suppose that $\tu \in \cC^2 (\R_+^2)$ is a function satisfying
$$
\div (|\mathbf{x}|^{-k} \nabla_\mathbf{x} \tu) = 0.
$$
Then it is clear that there is a $\tv \in \cC^2 (\R_+^2)$ such that
\begin{equation}\label{C-R-k}
|\mathbf{x}|^{-k} \nabla_\mathbf{x} \tu = {}^\perp \nabla_\mathbf{x} \tv.
\end{equation}
Hence $\nabla_\mathbf{x} \tu = |\mathbf{x}|^k {}^\perp \nabla_\mathbf{x} \tv$. Recalling that $(\mathbf{x}^\perp)^\perp = - \mathbf{x}$ we get ${}^\perp \nabla_\mathbf{x} \tu = - |\mathbf{x}|^k \nabla_\mathbf{x} \tv$ and finally
$$
\div (|\mathbf{x}|^k \nabla_\mathbf{x} \tv) = 0.
$$
We have shown the following fact.

\begin{Lem}
If $\tu \in \cC^2 (\R_+^2)$ is a solution to
\begin{equation}\label{eq:k}
\div (|\mathbf{x}|^{-k} \nabla_\mathbf{x} \tu) = 0,
\end{equation}
then there exists $\tv \in  \cC^2 (\R_+^2)$, related to $\tu$ by \eqref{C-R-k}, satisfying the dual equation
\begin{equation}\label{eq:dual}
\div (|\mathbf{x}|^k \nabla_\mathbf{x} \tv) = 0.
\end{equation}
Moreover, $\tv$ is determined up to a constant.
\end{Lem}

\begin{proof}
It remains to show that $\tv$ is determined up to a constant. Suppose that $\tv_1, \tv_2 \in \cC^2 (\R_+^2)$ satisfy \eqref{C-R-k}. Then $w := \tv_1 - \tv_2$ satisfies
$$
{}^\perp \nabla_\mathbf{x} w = 0,
$$
and therefore $w$ is a constant function.
\end{proof}

\begin{Rem}
In the case $k = 0$, equation \eqref{eq:k} reads as $\Delta_\mathbf{x} \tu = 0$, so that $\tu$ is a harmonic function on $\R_+^2$. Then, the dual equation is also $\Delta_\mathbf{x} \tv = 0$, however the above construction shows that $\tu + i \tv$ is a holomorphic function on the upper half-space. In this case, the existence of $\tv$ satisfying $\nabla_\mathbf{x} \tu = {}^\perp \nabla_\mathbf{x} \tv$ reads as Cauchy-Riemann equations.
\end{Rem}

\section*{Acknowledgements}
Tomasz Cieślak was partially supported by the National Science Centre, Poland (Grant No. 2017/26/E/ST1/00989).

\end{document}